\documentclass[a4paper,10pt]{article}
\usepackage{amsfonts, amsthm, amsmath, amssymb}

\newtheorem{Th}{Theorem}
\newtheorem{lem}{Lemma}
\newtheorem{prop}{Proposition}

\newcommand{\confrac}[2]{
\frac{\displaystyle{
\strut\hfill{#1}\hfill\;\vrule}}
{\displaystyle{
 \strut\vrule\;\hfill{#2}\hfill}}}

\title{On the Lenstra constant associated to the Rosen continued fractions }
\author{by \\ Hitoshi Nakada \thanks{supported by Bezoekersbeurs 
\textbf{B 61-620} of
the Nederlandse Organisatie voor Wetenschappelijk Onderzoek (NWO) 
and Grant-in Aid for 
Scientific research \textbf{18340032} 
of Japan Society for the Promotion of Science }
\\ {\small Department of Mathematics, Keio 
University} \\ {\small Hiyoshi, Kohoku-ku, Yokohama 223-8522, JAPAN} \\
{\small nakada@math.keio.ac.jp}}
\date{20 May 2007}

\begin{document}
\maketitle
\begin{abstract}
The purpose of this paper is to describe the relation between the Legendre and 
the Lenstra constants.  Indeed we show that they are equal whenever the 
Legendre constant exists; in particular, this holds for both 
Rosen continued fractions and $\alpha$-continued fractions.  
We also give the explicit value of the entropy of the Rosen map 
with respect to the absolutely continuous invariant probability measure. \\
{\bf Subject Classification}: 11K50, 37A45
\end{abstract}
\section{Introduction}
Let $x$ be an irrational number in $[0,\,1]$.  We denote by 
$\frac{p_n}{q_n}$ the n-th 
principal convergent of $x$ and recall the following.\\
\\
{\bf Theorem} (Legendre) {\it 
Suppose $p$ and $q$ ($>0$) are relatively prime integers and 
$| x \, - \, \frac{p}{q}|\,  < \, \frac{1}{2q^2}$. Then $\frac{p}{q}$ 
is a principal convergent to $x$. On the other hand, 
for any $c>\frac12$, there exist $x$ and $\frac{p}{q}$, which is not a 
principal convergent, such that 
$| x \, - \, \frac{p}{q}|\,  < \, c \, \frac{1}{q^2}$.}
 
In this sense, we call $\frac12$ the Legendre constant of the regular 
continued fractions.  
Now we consider the error of the principal convergents.  We put 
\[
    \Theta_{n} \, = \, q_{n}^2 \, |x \, - \, \frac{p_n}{q_n}| .
\]
The following fact was proved by \cite{Bo-J-W}.\\
\\
{\bf Theorem} (Bosma, Jager and Wiedijk, 1983) {\it For a.e. $x$, 
\[
 \lim_{n \to \infty} \, \frac1n \sum_{j=1}^{\infty} \, \{1 \le j \le n : 
 \Theta_{j} \le t\}
\]
exists for any $t$, $0 \le t \le 1$ and the limit is equal to
\[
     F(t) \, = \, \left\{ \begin{array}{lcr}  \frac{t}{\ln \, 2} & , 
     & 0\le t \le \frac12 \\
     {} & {} & \\
                                        \frac{1}{\ln \, 2}
               \,(1 \, - \, t \, + \, \ln\,2t) & , & \frac12 \le t \le 1
           \end{array} \right. .
\]}\\

We notice that $F(t)$ is linear in $0 \le t \le \frac12$ and not in 
$\frac12 \le t \le 1$.  
In this sense, we call $\frac{1}{2}$ the Lenstra constant of regular continued 
fractions because this fact was conjectured by 
H.~W.~Lenstra in 1981.  We can define Legendre constants and Lenstra constants 
for other types of continued fraction expansions in a similar manner  
(e.g. $\alpha$-expansions, \cite{Na}, \cite{Bo-J-W}).  In general, it is not 
hard to show that the Lenstra constant exists and is at least as large than 
the Legendre constant for each type of continued fraction expansion whenever 
the Legendre constant exists.  However a number of examples indicate that 
these constants seem to be equal to each other.       

Motivation of this paper is the metrical theory of Rosen continued fractions 
associated to Hecke groups.  It is possible to define the Legendre constants 
in this case even though the convergents of Rosen continued 
fractions are not rational numbers.   Indeed, in  1985 J.~Lehner claimed that 
the Legendre constant of Rosen continued fractions is greater than or equal 
to $\frac{1}{2}$ for Hecke group of any indices (\cite{Le2}). 
However, the proof was not correct (see \cite{Le3}, actually the correct value 
is less than $\frac12$) and a lower estimate was given by 
\cite{R-S} (where constants depend on the indices of Hecke groups).  On the 
other hand, the Lenstra constant was given by \cite{B-K-S} (and \cite{Na2} 
for even indices case).  
Also \cite{Na2} claimed (without proof) that the Lenstra constant and 
the Legendre constant are the same for each Hecke group of even index.
Here we note that Corollary 4.1. of \cite{B-K-S} 
did not say that the constant is the best possible one (which means it is 
the Lenstra constant), it is not hard to see that it is the best possible. 
We refer \cite{K-N-S} on this point. Indeed, the Lenstra constant for 
Rosen continued fraction is 
\[
\left\{ \begin{array}{cl}
\frac{\lambda}{\lambda + 2}, & q \,\, \text{even} \\
\frac{R}{R + 1}, & q \,\, \text{odd}
\end{array} \right.
\]

In the sequel, we show that the Lenstra constant is equal to the Legendre 
constant.  
In the next section, we introduce a generalized Diophantine approximation 
problem associated to a zonal Fuchsian group and give a law of large numbers 
for solutions 
of the Diophantine inequality.  This assures the existence of the Lenstra 
constant under the existence of the Legendre constant and also implies that 
the Lenstra constant is at least as large than Legendre constant.
In \S 3, we prove the equality of these two constants, mainly showing that 
the Lenstra constant can not be larger than the Legendre constant.  As a 
corollary (of the proof), we get the explicit value of the entropy of 
the Rosen map.  Finally 
we note that the same result holds for $\alpha$-continued fractions, 
$0 < \alpha \le \frac{1}{2}$.  For $\frac{1}{2} \le \alpha \le 1$, 
this result was shown by C.~Kraaikamp \cite{Kr} by a different way.  Recently, 
R.~Natsui \cite{Nt} showed that the existence of the Legendre constant for 
any $0 < \alpha \le \frac{1}{2}$ and thus we can apply the method of 
this paper to show the equality of these two constants. 
We stress the difference between concepts of two constants.  The Legendre 
constant is determined by the property which holds for all $x$, 
without exceptional point, on the other hand, the Lenstra constant comes 
from the metrical property which only holds for almost all points. 

\section{Generalized Diophantine Approximation}
Let $\Gamma$ be a finitely generated Fuchsian group acting on 
the upper half complex plane ${\bf H}^2$, $L$ the set of limit points 
of $\Gamma$ and $P$ the set of parabolic points. We assume that 
$\infty \in P$.  

An element $g \in \Gamma $ can be viewed as a $2\times 2$ real matrix
\[
\begin{pmatrix}
        a& b \\
        c& d 
  \end{pmatrix} 
\]
of determinant $1$.  We write $a=a(g)$, $b=b(g)$, $c(g)$ and $d=d(g)$.

J.~Lehner \cite{Le1} proved that there exists a positive number $t$ 
depending on $\Gamma$ such that
\[ 
\sharp \{\ g(\infty) : |x - g(\infty)| < \frac{t}{c^2 (g)}, \, 
     g \in \Gamma \} \, = \, \infty   
\]
for any $x \in L\setminus P$.  He also proved that if $\Gamma$
is of the first kind $(L = {\bf R})$, then for any sequence 
$\{ \, \varepsilon_n\} $ 
of positive numbers and a.e. $x \in L \setminus P $, 
there exists a sequence $\{\, g_n \} $ in $\Gamma$ such that 
\[
 |x - g_n(\infty)| \, < \, \frac{\varepsilon_n}{c^2 (g_n)}   
\]
Moreover, Patterson \cite{Pa} proved a kind of Khintchine theorem when $\Gamma$
is of the first kind: for example, his result implies that
\[ 
\sharp \{ \, g(\infty) : |x - g(\infty)| < 
            \frac{\ln |c(g)|}{c^2(g)},
                     \,  g \in \Gamma \} \, = \, \infty   
\]
for a.e. $x \in {\bf R} \setminus P $.

We shall estimate the asymptotic number of solutions of
\[
  g(\infty) : |x - g(\infty)| \, < \, \frac{t}{c^2 (g)}, 
     \, g \in \Gamma   
\]
for some positive real number $t$ and a.e. $ x \in {\bf R} \setminus P $.  
To do this, we consider a relation among 
the Diophantine inequality, geodesics of ${\bf H}^2 $, and 
geodesics of ${\bf H}^2 / \Gamma $.  We show that the ergodicity 
of the geodesic flow on ${\bf H}^2 / \Gamma $ with the 
hyperbolic measure is closely related to the quantitative theory 
of the Diophantine approximation on $\Gamma$, (see \cite{Su} for the 
qualitative theory).  The relation 
between the Diophantine inequality and geodesics of ${\bf H}^2$ 
also have been considered by A.~Haas \cite{Ha} and A.~Haas and C.~Series 
\cite{Ha-Ser} to 
discuss the Lagrange spectrum of the approximation on $\Gamma$. 
The {\it ``height"} of the $\Gamma$-congruent family of geodesics 
plays an important role in their discussion. 

This idea is also applicable to the theory of Diophantine 
approximations for complex numbers, where we have to consider ${\bf H}^3 $ 
\cite{Na3}.

Since $\infty \in P $, there exists 
\[ 
U_{\lambda} \, = \, 
       \begin{pmatrix}
        1& \lambda \\
        0& 1 
       \end{pmatrix}
  \in \Gamma , \,  \lambda \in {\bf R}_+  
\] 
such that
\[ 
    \{ {U_{\lambda}}^k : k \in {\bf Z} \} \, = \, \Gamma_{\infty} 
\]
where $\Gamma_\infty$ denotes the subgroup of $\Gamma $ that fixes $\infty$. 
We define the fundamental region $\cal F$ of $\Gamma $ by
\[  
{\cal F} = \{ z = x + i y : \frac{-\lambda}{2} < x
                  <   \frac{\lambda}{2}, \ y>0 \} \, 
     \bigcap \,  \left(\cap_{g \in \Gamma \setminus \Gamma_\infty}
\{ z : |c(g) \cdot z + d(g)| > 1 \}\right). 
\] 
${\mathcal F}$ is a hyperbolic polygon and its each side is an arc of 
the isometric circle of an element $g \in \Gamma$.  The image of this side 
by $g$ is also a side of ${\mathcal F}$, which is an arc of the 
isometric circle of $g^{-1}$.  We identify 
all such pairs and obtain  a hyperbolic surface.
It is well-known that the hyperbolic metric 
     $ds = \frac{\sqrt{{dx}^2 + {dy}^2}}{y}$
and the hyperbolic measure $d\mu = \frac{dx \, dy}{y^2}$ on ${\bf H}^2$
are invariant under $\Gamma$-action over ${\bf H}^2$. 

\begin{Th} 
Let 
\[ 
t_0 = \frac{1}{2} \, \min_
        {\scriptstyle g \in \Gamma \setminus \Gamma_\infty}
     |c(g)|,   
\]
then we have 
\[  \lim_{\scriptstyle N \to \infty}  
\frac{\sharp \{ g(\infty) : |x - g(\infty) < \frac{t}{{c^2}(g)}, 
      \  |c(g)| \le N, \ g \in \Gamma \} }{\ln N}
             =  \frac{4 \lambda \cdot  t } { \pi \cdot \mu ({\cal F})}
\]
for any $t$, $0<t< t_0$,  (a.e. $x \in {\bf R} \setminus P$).
\end{Th}

The proof of this theorem is basically the same as that of the main result 
in \cite{Na3} for the imaginary quadratic field case with the hyperbolic 
upper half space.  So we only give a sketch of the proof here.
We start with some lemmas.

We denote by $\gamma (x , y)$ the geodesic curve with the 
initial point $x$ and the terminal point $y$ for 
$(x , \beta ) \in ({\bf R} \cup \{ \infty \})^2 \setminus \{ \mbox{diagonal} 
\}$. 
We also denote by 
\[  
F_t(g(\infty)), \, t>0, 
\]
the circle which is tangent to the real line at $\frac{a(g)}{c(g)}$ with
radius $\frac{t}{c^{2}(g)}$ for $ g \notin \Gamma_\infty$ and \\
$ \{ x + iy : y = \frac{1}{2t} \} $ \ for $g \in \Gamma_\infty $.

It is possible to show the following:
\begin{lem}
If we fix $t>0$, then 
\[  
h(F_t (g(\infty))) = F_t (hg(\infty)) 
\]
for any h and g $\in \Gamma$.
\end{lem}

{\bf Proof.} \ This follows from the fact that $F_t(g(\infty))$ 
is an image of 
$ \{ x + iy : y = \frac{1}{2t} \} $, which makes an invariant family 
of circles under $\Gamma$-action. \hfill  \qed
\begin{lem}
For any $k>0$,
\[  
|x - g(\infty)| < \frac{t}{{c^2}(g)} 
\]
holds if and only if 
$\gamma (\infty, x)$ and $\cap F_t (g(\infty))$ 
do not cross to each other.
\end{lem}

If $t < t_0$, then we see that $\{ F_t(g(\infty)) \}$ is 
a disjoint family of circles, that is, 
\[ 
  F_t (g(\infty)) \cap F_t (h(\infty)) = \emptyset 
  \]
if $g(\infty) \neq h(\infty)$.  Thus we have the following:

\begin{lem} 
If $0 < t < t_0$, then every point of $F_t(g(\infty)) 
\setminus ({\bf R} \cup \{ \infty \} )$ is congruent to some point of 
$F_t(\infty) \cap ({\it F} \setminus \{ \infty \} )$.  In particular,  
if $ p \in  F_t(g(\infty)) \setminus ({\bf R} \cup \{ \infty \}) $, 
then there exists $ h \in \Gamma $ such that $ h(p) = x + iy, 
\frac{- \lambda}{2} < x < \frac{\lambda}{2}$ and 
$ y = \frac{1}{2t} $.  
\end{lem}
\underline {\bf Proof of the theorem.} 
Let ${\bf T}({\bf H}^2)$ and {\bf T}(${\cal F}$) be the unit tangent 
bundles of ${\bf H}^2$ and ${\cal F}$, respectively.  We consider the 
geodesic flows $f_s$ and ${\hat f}_s$ on 
${\bf T}({\bf H}^2)$ and {\bf T}(${\cal F}$), respectively. 
For ${\omega}^{\ast} \in {\bf T}({\bf H}^2)$, there is a unique 
geodesic $(x , \beta )$  passing tangentially through ${\omega}^{\ast}$.
If $x \neq \infty$ and $\beta \neq \infty$, then we denote by 
$s$ the (signed) hyperbolic length from the top of the geodesic arc
$(x , \beta )$ to $\omega$, which is the base point of 
${\omega}^{\ast}$.  If $x = \infty$ (or $\beta = \infty $), then 
we denote by $s$ the hyperbolic length from the point $\beta + i$ 
or ($x + i$) to $\omega$, (respectively).  Thus we can 
parameterize ${\omega}^{\ast} \in {\bf T}({\bf H}^2)$ by 
$(x,\beta,s)\in(({\bf R}\cup\{\infty\})^2 \setminus \{diagonal\})
\times{\bf R}$.
So if $0<t<t_0$, we see from Lemmas 2 and 3 that
\begin{eqnarray*}
\left| \, \sharp \{ g(\infty) : |x - g(\infty)| < \frac{t}{{c^2}(g)},
                   \ |c(g)|  \le N, \ g \in \Gamma \} \right. \, -  \\
  \left. 
  \sharp \left\{ s :  \begin{array}{l}
            f_s (\infty , x , -\ln (4t_0 + 1))
           \ \mbox{crosses a circle} \ F_t (g(\infty)) \ \mbox{from}  \\
             \mbox{outside at time} \ s, 
          \ 0<s \le \ln(4 t_0 +1) - \ln t + 2 \ln N 
        \end{array} \right\}   \, \right| 
        \,  \le \, 1 
\end{eqnarray*}
and 
\begin{eqnarray*} 
{} & {} & \sharp \left\{ s :  \begin{array}{l}
            f_s (\infty , x , -\ln (4t_0 + 1))
           \ \mbox{crosses a circle} \ F_t (g(\infty)) \ \mbox{from}  \\
             \mbox{outside at time} \ s, 
          \ 0<s \le \ln(4 t_0 +1) - \ln t + 2 \ln N 
        \end{array} \right\} \\ 
{} & = & \sharp \left\{ s : \begin{array}{l}
       {\hat f}_s({\omega}^*) \ \mbox{crosses} 
       \, F_t(\infty)
                     \ \mbox{from below at time} \ s, \\
          0 < s \le \ln (4 t_0 + 1) - \ln t + 2 \ln N 
        \end{array} \right\} 
\end{eqnarray*}
where $\omega^{\ast} \in {\bf T}({\cal F})$ is the point corresponding to 
$(\infty , x , -\ln (4t_0 + 1) ) \in {\bf T}({\bf H}^2)$.

Now we apply the individual ergodic theorem for 
$({\bf T}({\cal F}), \hat f , \hat {\mu} )$ to our problem.  Here, the 
hyperbolic measure $\hat {\mu}$ on ${\bf T}({\cal F})$ induced from  
$\mu$ is defined by
\[
     \hat{\mu} = \frac{dx\ d\beta\ ds}{{(x - \beta)}^2}
\]
if we parameterize a point in ${\bf T}({\cal F})$ by 
$(x , \beta , s)$.

\begin{prop} 
If we fix $t$, $0<t<t_0$, then
\[
\lim_{u \to \infty} \frac{ \sharp \{ s : 
f_s({\omega}^{\ast}) \, \mbox{crosses} \, F_t(\infty) \ \mbox{from below}, 
\ 0<s<u \} }{u}
                 =
\frac{\mu \{ x+iy \in {\cal F} : y>\frac1{2t}\}}{\pi \cdot \mu({\cal F})}
\]
for a.e. $\omega^{\ast} \in {\bf T}({\cal F}).$ 
\end{prop}
Moreover, by using an approximation method, on $t$, we have
\begin{prop}
For a.e. $\omega^{\ast} \in {\bf T}({\cal F}),$ 
\[
\lim_{u \to \infty} \frac{ \sharp \{ s : 
f_s({\omega}^{\ast}) \, \mbox{crosses} \, F_t(\infty)\ \mbox{from below}, 
\ 0<s<u \} }{u}
                 =
\frac{\mu \{ x+iy \in {\cal F} : y>\frac1{2t}\}}{\pi \cdot \mu({\cal F})}
\]
for any $t$, $0<t<t_0.$
\end{prop}
Furthermore, it is possible to show that if $\omega^{\ast} =(x , \beta , s) 
\in {\bf T}({\cal F})$ has the above property, then for any  $x '
\in {\bf R} \cup \{\infty\}$ and $s' \in {\bf R}$, \ $\omega^{\ast \ast} 
=(x' , \beta , s')$ 
also has the same property.  Since the hyperbolic length between 
$x + (4t_0 + 1)i $ and $x + \frac1{N}i$ is equal to 
$\ln N + \ln (4t_0 + 1)$, we have
\begin{eqnarray*}
& \quad & \lim _{N \to \infty}
\frac{\sharp \{ g(\infty) : |x - g(\infty)| < \frac{t}{{c^2}(g)},
                \ |c(g)|  \le N, \ g \in \Gamma \}}{\ln N} \\ 
& = &  2 \cdot \frac{\mu \{ x+iy \in {\cal F} : y>\frac1{2t}\}}
                    {\pi \cdot \mu({\cal F})}  \\
& = & \frac{4 \lambda \cdot t}{\pi \cdot \mu({\cal F})} 
\end{eqnarray*}
for any $t$, $0<t<t_0$ and a.e. 
$x \in {\bf R} \setminus  P$.
\qed
\\[0.5cm]
\underline{\bf Some remarks.} 
We apply the theorem to Hecke groups of index $k$, $G_k$,
$ 3 \le k \le \infty$, and its congruent subgroups $G_k(m)$.
   Here $G_k$ is the group generated by
\[
    \begin{pmatrix}
        0& 1 \\
       -1& 0 
    \end{pmatrix}
         \quad      \text{and}    \quad
    \begin{pmatrix}
        1& \lambda_k \\
        0& 1 
    \end{pmatrix} 
\]
where $\lambda_k = 2 \cdot \cos \frac{\pi}{k}$ for $k \geq 3$ 
(and $ = 2$ when $k= \infty$), 
and $G_k(m)$ the subgroup of $G_k$ defined by 
\[
 G_k(m) = \left\{g \in G_k : g \equiv 
\begin{pmatrix} 
    \pm 1 & 0 \\ 
    0 & \pm 1 
 \end{pmatrix}
\,  \bmod \,(m \cdot \lambda_k)\right\}
\]
where $(m \cdot \lambda_k)$ denotes the ideal generated by 
$m \cdot \lambda_k$ with a positive integer $m$.  

A fundamental region ${\cal F}_k$ of $G_k$ is given by 
\[
{\cal F}_k = \{ x+iy : 
            -\cos {\pi}{n}< x \le \cos {\pi}{k}, \ x^2+y^2>1, \, y>0\}.
\]
Thus we see that $G_k$ is of the first kind and 
\[
   P = P_k = G_k(\infty) = \{ g(\infty) : g \in G_k \}
\]
if $k \ne \infty$.  In this case, we have 
\[
\lim_{N \to \infty}
\frac{\sharp \{ g(\infty) : |x - g(\infty)| < \frac{t}{{c^2}(g)},
                    \ |c(g)|  \le N, \ g \in G_k \}}{\ln N}  
 =  \frac{4 \cdot k \cdot \lambda_k \cdot t}{(k-2) \cdot \pi^2 } 
\]
for any $t$, $0<t< \frac{1}{2}$ and a.e. $x \in {\bf R}
\setminus P_k$.  
We can also apply Theorem 1 to $G_k(m)$.  Then, in this case, the set of 
parabolic points of $G_k$ is 
divided into a finite number of disjoint sets.  We put 
\[
t_m = \frac{1}{2} 
\min_{g \in G_k(m) \setminus G_k(m)_\infty}  |c(g)| . 
\]
There exists a constant $C > 0$ such that  
\[
\lim_{\scriptstyle N \to \infty}
\frac{\sharp \{ g(\infty) : |x - g(\infty)| < \frac{t}{{c^2}(g)},
                    \ |c(g)|  \le N, \ g \in G_k(m)  \}}{\ln N}  
 = C \cdot t
\]
for any $t$, $0 < t < t_m$, and a.e. $x \in {\bf R}$. Since $G_k(m)$ 
is a subgroup of $G_k$, for each cusp of the fundamental region of 
$G_k(m)$ there exists $g_{\eta} \in G_k$ such that $g_{\eta}(\eta) = 
\infty$. It is obvious that $g_{\eta}{\mathcal F}$ is a fundamental 
region of $g_{\eta}G_k(m) g_{\eta}^{-1}$. Since $G_k(m)$ is normal, 
$g_{\eta} {\mathcal F}$ is a fundamental region of 
$G_k(m)$ .  This means the ``width" of the cusp $\eta$ is the same as 
that of $\infty$.  Thus we have 
\[
\lim_{\scriptstyle N \to \infty}
\frac{\sharp \{ g(\infty) : |x - g(\infty)| < \frac{t}{{c^2}(g)},
                    \ |c(g)|  \le N, \ g \in G_k, \,\, 
                    \exists \hat{g} \in G_k(m) \, \text{s.t.} \, 
                    g(\infty) = \hat{g}(\eta) 
                      \}}{\ln N}  
 = C \cdot t
\]
for any $t$, $0 < t < t_m$, and a.e. $x \in {\bf R}$.  Moreover, it turns 
out that $t_m \to \infty$ as $m \to \infty$.  This shows the following : 
\\
{\bf Corollary} {\it For a.e. $x \in {\bf R}$ 
\[
\lim_{\scriptstyle N \to \infty}
\frac{\sharp \{ g(\infty) : |x - g(\infty)| < \frac{t}{{c^2}(g)},
                    \ |c(g)|  \le N, \, g \in G_k \}}{\ln N}  
 =  \frac{4 \cdot k \cdot \lambda_k \cdot t}{(k-2) \cdot \pi^2 } 
\]
for any $t>0$.} \\[0.3cm]
\underline{\bf Remark.}  The above proof (of this corollary) shows the 
equidistributed property (a.e.) of solutions associated to cusps. 
We refer R.~Moeckel \cite{Mo} for the original idea of this method.

\section{Rosen Continued Fractions }
Given any element of $G_k$ of the form 
\[
\begin{pmatrix}
p & \cdot \\ q & \cdot 
\end{pmatrix} ,
\]
we have $g(\infty) = p/q$ and, moreover for any $\tilde g \in
G_k$ with $\tilde g (\infty) = p/q$ we have 
\[
  \tilde g \, = \, \begin{pmatrix} p & \cdot \\ q & \cdot \end{pmatrix}
\qquad {\mbox or} \qquad \begin{pmatrix} -p & \cdot \\ -q & \cdot \end{pmatrix}
\]
So, for any parabolic point of $G_k$, $p$ and $q>0$ are uniquely determined.

We define the $\lambda_k$-nearest continued fraction transformation of 
$[-\frac{\lambda_k}{2}, \, \frac{\lambda_k}{2})$ onto itself by 
\[
           S(x)) \, = \left\{ \begin{array}{ccl} 
           \left| \frac{1}{x} \right| \, - \, 
           \left[\, \left| \frac{1}{x} \right| \, \right]_k  
           & \mbox {for} & x \ne 0 \\ 
           {} & {} & {} \\
           0 & \mbox{for} & x = 0 
\end{array} \right.
\]
where $[w]_k \, = \, b \cdot \lambda_k$ ; $b \in \bf{Z}$, when 
$w \in [b-\frac{\lambda_k}{2}, \, b+\frac{\lambda_k}{2})$.  
We put 
\[
\varepsilon_n \, = \varepsilon_n (x) \, = \, \mbox{sgn} S^{n-1}(x)
\]
and 
\[
a_n \, = \, a_n(x)\,= \, [\, \left|\frac{1}{S^{n-1}(x)} \right| \, ]_k
\]
for any $n \ge 1$ and have a continued fraction expansion:
\[
x \, = \,   \confrac{\varepsilon_1}{a_1} \, + \, 
            \confrac{\varepsilon_2}{a_2} \, + \, 
            \confrac{\varepsilon_3}{a_3} \, + \,  \cdots
\]
We call this expansion the Rosen continued fraction expansion of $x$. 
In general, if a continued fraction, either finite or infinite, is 
given by some $x$ as its Rosen continued fraction expansion, we call it 
a Rosen continued fraction.
We define the principal convergent $\frac{p_n}{q_n}, \, n \ge 0$, by 
\[
\begin{pmatrix}
p_{-1} & p_0 \\
q_{-1} & q_0 
\end{pmatrix} 
\, = \, 
\begin{pmatrix}
1  &  0 \\
0  &  1 
\end{pmatrix} 
\]
and 
\[
\begin{pmatrix}
p_{n-1} & p_n \\
q_{n-1} & q_n
\end{pmatrix} 
\, = \, 
\begin{pmatrix}
0   & \varepsilon_1 \\
1   & a_1
\end{pmatrix} 
\begin{pmatrix}
0   & \varepsilon_2 \\
1   & a_2
\end{pmatrix} 
\cdots
\begin{pmatrix}
0   & \varepsilon_n \\
1   & a_n
\end{pmatrix} 
\qquad \mbox{for} \, n \ge 0 . 
\]
It is easy to see that $q_n > 0$ for any $n \ge 0$. 
\begin{lem} We have
\[\begin{aligned}
\dfrac{1}{q_n \left(q_{n+1} + q_n\right)} \, &\le \, \left| x \, - \, \frac{p_n}{{q_n}} \right|  \\
\\
& \le \,
\left\{ \begin{array}{cl}
\dfrac{1}{q_n^2\left( 1 - \dfrac{\lambda}{2}\right)} & \mbox{if $k$ is even}; 
\\
\\
\dfrac{1}{q_n^2 \left( \dfrac{1}{R} - \dfrac{\lambda}{2}\right)} & \mbox{otherwise}, \end{array}\right.
\end{aligned}
\]
where $R$ is the positive root of
$R^2 + (2 - \lambda) R - 1 = 0$.
\end{lem}
\begin{proof}
We have
\[
T^n(x) \, = \,
\begin{pmatrix} p_{n-1} & p_{n} \\ q_{n-1} & q_{n} \end{pmatrix}^{-1} (x)
\]
and
\begin{eqnarray*}
x & = &
\begin{pmatrix} p_{n-1} & p_{n} \\ q_{n-1} & q_{n} \end{pmatrix} (T^{n} x) \\
\\
{} & = &
\frac{p_{n-1} T^{n}x \, + \, p_n}{q_{n-1} T^{n}x \, + \, q_n} .
\end{eqnarray*}
Thus we see
\begin{eqnarray*}
\left| x \, - \, \frac{p_n}{{q_n}} \right| & = &
\left| \frac{p_{n-1} T^{n}x \, + \, p_n}{q_{n-1} T^{n}x \, + \, q_n}\, - \,
                     \frac{p_n}{q_n} \right|  \\
{} & = & \left| \frac{T^n x}{q_n \left(q_{n-1} T^nx + q_n\right)} \right| \\
{} & = & \left| \frac{1}{q_n^2 \left( \frac{q_{n-1}}{q_n} +
\frac{1}{T^n x} \right) }\right|  \\
{} & = &
\left| \frac{1}{q_n^2 \left( \frac{q_{n-1}}{q_n} + \varepsilon_{n+1}(x)
         \left( r_{n+1} \lambda + T^{n+1} x \right) \right)} \right| \\
{} & = &
\frac{1}{q_n^2} \, \frac{1}{\frac{q_{n+1}}{q_n} +  T^{n+1}x }
\end{eqnarray*}
Since $\frac{q_{n+1}}{q_n} > 1$ if $k$ is even 
(and $\frac{q_{n+1}}{q_n} > \frac{1}{R}$ if $k$ is odd, respectively), 
the result follows.
\end{proof}
\begin{lem}
For a.e. $x \in {\mathbb I}$, 
\[
\lim_{n \to \infty} \frac{1}{n} \ln q_n 
\]
exists and is equal to the half of the entropy $h_k$ of the Rosen map w.r.t. 
the absolutely continuous invariant probability measure. 
\end{lem}
\begin{proof} Let $\mu$ be the absolutely continuous invariant probability 
measure
for $T$. Since $(T, \mu)$ is ergodic (see \cite{B-K-S}), we have from 
Shannon-McMillan-Breiman's theorem that the entropy $h_k$ of the Rosen map 
is given by 
\[
h_k \, = \, 
\lim_{n \to \infty} 
\frac{1}{n} \ln \mu( \Delta[\varepsilon_{1} \, : \,  r_1\, ,
\,
\ldots \, , \, \varepsilon_{n} \, : \, r_n] )  \quad \text{a.e.} 
\]
We can replace $\mu$ to the normalized Lebesgue measure $m$ because 
$\mu$ has a positive density function bounded away from 0 and bounded from 
above, see \cite{B-K-S}, that is, 
\[
h_k \, = \, 
\lim_{n \to \infty} 
\frac{1}{n} \ln m( \Delta[\varepsilon_{1} \, : \,  r_1\, ,
\,
\ldots \, , \, \varepsilon_{n} \, : \, r_n] )  \quad \text{a.e.} 
\]
From Lemma 4, it turns out that 
\[
\lim_{n \to \infty} 
\frac{1}{n} \ln m( \Delta[\varepsilon_{1} \, : \,  r_1\, ,
\,
\ldots \, , \, \varepsilon_{n} \, : \, r_n] )  
\, = \, 2 \lim_{n \to \infty} \frac{1}{n} \ln q_n
\]
if the limit of the left hand side exists. Thus we have 
\[
\lim_{n \to \infty} \frac{1}{n} \ln q_n
\, = \, \frac{h_k}{2}
\]
for a.e. $x \in {\mathbb I}$. 
\end{proof}

Now we denote by ${\mathcal L}g_{k}$ the Legendre constant of Rosen continued 
fractions of index $k$, that is, the following hold : \\
(1) for $c \le {\mathcal L}g_{k}$, and $x \in [-\frac{\lambda_k}{2}, \, 
\frac{\lambda_k}{2})$, and $\begin{pmatrix} p & \cdot \\ q & \cdot 
\end{pmatrix} \in G_k$, $q \ne 0$, if $\left| x \, - \, \frac{p}{q} 
\right| < \frac{c}{q^2}$ holds, then $\frac{p}{q} = \frac{p_n}{q_n}$ 
for some $n \ge 0$, \\
(2) for $c > {\mathcal L}g_{k}$, there exist $x \in [-\frac{\lambda_k}{2}, \, 
\frac{\lambda_k}{2})$ and $\begin{pmatrix} p & \cdot \\ q & \cdot 
\end{pmatrix} \in G_k$, $q \ne 0$, such that 
$\left| x \, - \, \frac{p}{q} \right| < \frac{c}{q^2}$ and 
$\frac{p}{q} \ne \frac{p_n}{q_n}$ for any $n \ge 0$. \\
As mentioned in the introduction, the existence of ${\mathcal L}g_{k}$ was 
shown in \cite{R-S}.  

On the other hand, we denote by $Le_{k}$ the Lenstra constant of Rosen 
continued fractions of index $k$. This means that for almost every 
$x \in [-\frac{\lambda_k}{2}, \, \frac{\lambda_k}{2})$, 
\[
\lim_{N \to \infty} \frac{1}{N} 
\sharp\left\{ 
n \, : \, 1 \le n \le N, \, \Theta_n(x) < t \right\} \, = \, 
C_k \cdot t \quad \mbox{for any} \quad 0 < t \le Le_{k} ,
\]
where $C_k$ is an absolute constant, which is given in \cite{B-K-S}, and 
\[
    \Theta_{n} \, = \, q_{n}^2 \, |x \, - \, \frac{p_n}{q_n}| .
\]
We will prove the following : 
\begin{Th}
For any $k \, \ge \, 3$, we have ${\mathcal L}g_{k} \, = \, Le_{k}$. 
\end{Th}

To prove this theorem, we will show the following two propositions. 
The assertion of Theorem 2 is a direct consequence of these two. 
\begin{prop}
For any $k \ge 3$, we have ${\mathcal Lg_k} \le {\mathcal Le_k}$. 
\end{prop}
\noindent
{\bf Proof.}  
Suppose $0 \le t \le {\mathcal L}g_{k}$.  From Corollary of \S 2, we have 
for a.e. $x \in [-\frac{\lambda_k}{2}, \, \frac{\lambda_k}{2})$ 
\begin{eqnarray*}
{} & {} & \lim_{N \to \infty} 
\frac{\sharp \{ g(\infty) \, : \, |x - g(\infty)| < \frac{t}{{c^2}(g)},
                    \ |c(g)|  \le q_n,  \, g \in G_k, \, \mbox{for some}\, 
                     0 \le n \le N \}}
                    {\ln q_N}  \\
{}& = & 
\lim_{N \to \infty} 
\frac{\sharp \{ g(\infty) \, : \, |x - g(\infty)| < \frac{t}{{c^2}(g)},
                    \ |c(g)| \le q_N, \ g \in G_k \}}{\ln q_N}  \\
{} & = &
 \frac{4 \cdot k \cdot \lambda_k \cdot t}{(k-2) \cdot \pi^2 }
\end{eqnarray*}
We note the following 
\begin{eqnarray*}
{}&{}& \lim_{N \to \infty} 
\frac{\sharp \{ g(\infty) \, : \, |x - g(\infty)| < \frac{t}{{c^2}(g)},
                    \ |c(g)|  \le q_n,  \, g \in G_k, \, \mbox{for some}\, 
                     0 \le n \le N \}}
                    {\ln q_N} \\
& = & 
\lim_{N \to \infty} 
\frac{\frac{1}{N}\sharp \{ g(\infty) \, : \, |x - g(\infty)| < 
\frac{t}{{c^2}(g)},
                    \ |c(g)|  \le q_n,  \, g \in G_k, \, \mbox{for some}\, 
                     0 \le n \le N \}}
                    {\frac{1}{N}\ln q_N} 
\end{eqnarray*}
From Lemma 5, the denominator of the right hand side converges to 
$\frac{h_k}{2}$ (a.e.), 
we see that the numerator converges (a.e.) to 
\[
\frac{4 \cdot k \cdot \lambda_k \cdot t \cdot h_k}{2(k-2) \cdot \pi^2 }
\]
This means 
\begin{equation}
\lim_{N \to \infty} \frac{1}{N}
\sharp\{ n \, : \, 1 \le n \le N, \, \Theta < t\} 
\, = \, 
\frac{4 \cdot k \cdot \lambda_k \cdot t \cdot h_k}{2 (k-2) \cdot \pi^2 }
\qquad \mbox{a.e.} 
\end{equation}
for $0 \le t \le {\mathcal L}g_{k}$. \qed
\begin{prop}
For any $k \ge 3$, we have ${\mathcal Lg_k} \ge {\mathcal Le_k}$. 
\end{prop}
\noindent
{\bf Proof.} Suppose that $t > {\mathcal Lg_k}$.  Then there exist $x \in
[-\frac{\lambda_k}{2}, \, \frac{\lambda_k}{2})$ and 
$\begin{pmatrix} p & \cdot \\ q & \cdot \end{pmatrix} \in G_k$ 
such that 
\begin{equation}
\left| x \, - \, \frac{p}{q} \right| \, < \, \frac{t}{q^2} 
\qquad \mbox{and} \qquad \frac{p}{q} \ne \frac{p_n}{q_n} \,\,
\mbox{for any} \, \, n \ge 0.
\end{equation}
From this inequality, there exists $\varepsilon > 0$ such that 
\[
\left|x \, - \, \frac{p}{q}\right| \, < \, 
\frac{t}{q^2} - \varepsilon .
\]
If $y \in [-\frac{\lambda_k}{2}, \, \frac{\lambda_k}{2})$ is sufficiently 
close to $x$, i.e. 
\begin{equation}
|x - y| < \frac{\varepsilon}{2} ,
\end{equation}
then 
\begin{equation}
|y - \frac{p}{q}| < \frac{t}{q^2} \, - \, \frac{\varepsilon}{2} ,
\end{equation}
holds. Moreover there exists a positive integer $M_0$ such that 
\[
(\varepsilon_i(x), a_i(x)) \, = \, (\varepsilon_i(y), a_i(y)) \qquad 
 \mbox{for} \qquad 1 \le i \le M_0
\]
implies that (2) holds. 

Now we look at the expansion of $\frac{p}{q}$ and $x$.  Since $\frac{p}{q}$ 
is a parabolic point of $G_k$, it is easy to see that $\frac{p}{q}$ has a 
finite Rosen expansion, say, 
\[
\frac{p}{q} \, = \, 
\confrac{\hat{\varepsilon}_1}{\hat{a}_1} \, + \, 
            \confrac{\hat{\varepsilon}_2}{\hat{a}_2} \, + \, 
       \cdots \, + \, \confrac{\hat{\varepsilon}_n}{\hat{a}_n}.
\]
This means there exists $m$, $1 \le m \le n$ such that 
\[
(\varepsilon_i, a_i) \, = \, (\hat{\varepsilon}_1, \hat{a}_i)
\quad \mbox{for} \,\, 1 \le i \le m-1 \quad \mbox{and} \quad
(\varepsilon_m, a_m) \, \ne \, (\hat{\varepsilon}_m, \hat{a}_m)
\]
with 
\[
x \, = \,   \confrac{\varepsilon_1}{a_1} \, + \, 
            \confrac{\varepsilon_2}{a_2} \, + \, 
            \confrac{\varepsilon_3}{a_3} \, + \,  \cdots
            \confrac{\varepsilon_m}{a_m} \, + \, 
            \confrac{\varepsilon_{m+1}}{a_{m+1}} \, + \, \cdots 
            \, + \, 
            \confrac{\varepsilon_n}{a_n} \, + \,  \cdots
\]
Here we may assume that $M_0 > n$.  
We choose a positive integer $M_1$ sufficiently large and define 
\[
\hat{a}_0 \, = \, M_1 \cdot \lambda_{k} 
\]
We see that 
\[
\confrac{\hat{\varepsilon}_0}{\hat{a}_0} \, + \, 
\confrac{\hat{\varepsilon}_1}{\hat{a}_1} \, + \, 
            \confrac{\hat{\varepsilon}_2}{\hat{a}_2} \, + \, 
       \cdots \, + \, \confrac{\hat{\varepsilon}_n}{\hat{a}_n}
        \, + \, \confrac{\hat{\varepsilon}_{n}}{\hat{a}_{n}}  
\]
is a Rosen continued fraction, where $ \hat{\varepsilon}_0 = +1$. 
We fix any finite Rosen continued fraction 
\[
\confrac{\varepsilon'_1}{b_1} \, + \, 
            \confrac{\varepsilon'_2}{b_2} \, + \, \cdots \, + \, 
            \confrac{\varepsilon'_l}{b_l}  
\]
so that 
\[
\confrac{\varepsilon'_1}{b_1} \, + \, 
            \confrac{\varepsilon'_2}{b_2} \, + \, \cdots \, + \, 
            \confrac{\varepsilon'_l}{b_l} 
            \, + \, 
\confrac{\hat{\varepsilon}_0}{\hat{a}_0} \, + \, 
\confrac{\hat{\varepsilon}_1}{\hat{a}_1} \, + \, 
            \confrac{\hat{\varepsilon}_2}{\hat{a}_2} \, + \, 
       \cdots \, + \, \confrac{\hat{\varepsilon}_n}{\hat{a}_n}
        \, + \, \confrac{\hat{\varepsilon}_{n}}{\hat{a}_{n}}  
\]
and 
\[
\confrac{\varepsilon'_1}{b_1} \, + \, 
            \confrac{\varepsilon'_2}{b_2} \, + \, \cdots \, + \, 
            \confrac{\varepsilon'_l}{b_l} 
            \, + \, 
\confrac{\hat{\varepsilon}_0}{\hat{a}_0} \, + \, 
\confrac{\varepsilon_1}{a_1} \, + \, 
            \cdots
            \confrac{\varepsilon_m}{a_m} \, + \, 
            \confrac{\varepsilon_{m+1}}{a_{m+1}} \, + \, \cdots 
            \, + \, 
            \confrac{\varepsilon_n}{a_n} \, + \, \cdots \,+ \, 
            \confrac{\varepsilon_{M_0}}{a_{M_0}} \, + \, \cdots
\]
are also Rosen continued fractions.  Suppose that 
$y_0 \in [-\frac{\lambda_k}{2}, \, \frac{\lambda_k}{2})$ with the Rosen 
expansion of the form 
\[
\confrac{\varepsilon'_1}{b_1} \, + \, 
            \confrac{\varepsilon'_2}{b_2} \, + \, \cdots \, + \, 
            \confrac{\varepsilon'_l}{b_l} 
            \, + \, 
\confrac{\hat{\varepsilon}_0}{\hat{a}_0} \, + \, 
\confrac{\varepsilon_1}{a_1} \, + \, 
            \cdots
            \confrac{\varepsilon_m}{a_m} \, + \, 
            \confrac{\varepsilon_{m+1}}{a_{m+1}} \, + \, \cdots 
            \, + \, 
            \confrac{\varepsilon_n}{a_n} \, + \, \cdots \,+ \, 
            \confrac{\varepsilon_{M_0}}{a_{M_0}} \, + \, \mbox{(free)}
\]
and 
\[
\frac{P}{Q} \, = \, 
\confrac{\varepsilon'_1}{b_1} \, + \, 
            \confrac{\varepsilon'_2}{b_2} \, + \, \cdots \, + \, 
            \confrac{\varepsilon'_l}{b_l} 
            \, + \, 
\confrac{\hat{\varepsilon}_0}{\hat{a}_0} \, + \, 
\confrac{\hat{\varepsilon}_1}{\hat{a}_1} \, + \, 
            \confrac{\hat{\varepsilon}_2}{\hat{a}_2} \, + \, 
       \cdots \, + \, \confrac{\hat{\varepsilon}_n}{\hat{a}_n}
        \, + \, \confrac{\hat{\varepsilon}_{n}}{\hat{a}_{n}} . 
\]
We note that $P$ and $Q$ are uniquely determined by 
$\begin{pmatrix} P & \cdot \\ Q & \cdot 
\end{pmatrix} \in G_k$ and $Q > 0$.  
We also note that $y := S^{l+1}(y_0)$ has Rosen continued fraction 
\[
\confrac{\varepsilon_1}{a_1} \, + \, 
            \cdots
            \confrac{\varepsilon_m}{a_m} \, + \, 
            \confrac{\varepsilon_{m+1}}{a_{m+1}} \, + \, \cdots 
            \, + \, 
            \confrac{\varepsilon_n}{a_n} \, + \, \cdots \,+ \, 
            \confrac{\varepsilon_{M_0}}{a_{M_0}} \, + \, \mbox{(free)}
\]
which satisfies (2).  
We put 
\[
\begin{pmatrix} P_{n+l} & P_{n+l+1} \\ Q_{n+l} & Q_{n+l+1} 
\end{pmatrix} 
\, = \, 
\begin{pmatrix} 0 & \varepsilon'_1 \\ 1 & b_1 \end{pmatrix} 
\cdots 
\begin{pmatrix} 0 & \varepsilon'_l \\ 1 & b_l \end{pmatrix} 
\begin{pmatrix} 0 & \hat{\varepsilon}_0 \\ 1 & \hat{a}_0 \end{pmatrix} 
\begin{pmatrix} 0 & \hat{\varepsilon}_1 \\ 1 & \hat{a}_1 \end{pmatrix} 
\cdots
\begin{pmatrix} 0 & \hat{\varepsilon}_n \\ 1 & \hat{a}_n \end{pmatrix} , 
\]
which implies $(P,Q) = (P_{n+l+1}, Q_{n+l+1})$, and estimate  
\[
\left| y_0 \, - \, \frac{ P_{n+l+1}}{Q_{n+l+1}} \right|
\]
We also define 
\[
\begin{pmatrix} P_{l-1} & P_{l} \\ Q_{l-1} & Q_{1} 
\end{pmatrix} 
\, = \, 
\begin{pmatrix} 0 & \varepsilon'_1 \\ 1 & b_1 \end{pmatrix} 
\cdots 
\begin{pmatrix} 0 & \varepsilon'_l \\ 1 & b_l \end{pmatrix} 
\]
and 
\[
\begin{pmatrix} P_{l} & P_{l+1} \\ Q_{l} & Q_{1+1} 
\end{pmatrix} 
\, = \, 
\begin{pmatrix} 0 & \varepsilon'_1 \\ 1 & b_1 \end{pmatrix} 
\cdots 
\begin{pmatrix} 0 & \varepsilon'_l \\ 1 & b_l \end{pmatrix} 
\begin{pmatrix} 0 & \hat{\varepsilon}_0 \\ 1 & \hat{a}_0 \end{pmatrix} 
\]
We denote by $U$ the linear fractional transformation defined by 
$\begin{pmatrix} P_{l} & P_{l+1} \\ Q_{l} & Q_{l+1} \end{pmatrix}$. 
Then it is easy to see that 
\[
U\left(\frac{p}{q}\right) \, = \, \frac{P}{Q} \quad \mbox{and} \quad 
U(y) \, = \, y_0
\]
Thus 
\[
\left| y_0 \, - \, \frac{ P_{n+l+1}}{Q_{n+l+1}} \right| \, = \, 
\left| U\left(\frac{p}{q}\right) \, - \, U(y) \right|
\]
and the following holds: 
\begin{eqnarray*}
{} & {} & 
\left| U\left(\frac{p}{q}\right) \, - \, U(y) \right| \\
{} & = & 
\left| 
\frac{P_{l} y + P_{l+1}}{Q_{l} y + Q_{l+1} } \, - \, 
\frac{P_{l} \frac{p}{q} + P_{l+1}}{Q_{l}\frac{p}{q}  + Q_{l+1} }
\right|  \\
{} & = & 
\left| 
\left(
\frac{P_{l} y + P_{l+1}}{Q_{l} y + Q_{l+1} } \, - \, \frac{P_{l+1}}{Q_{l+1}}
\right)
\, + \, 
\left(
\frac{P_{l+1}}{Q_{l+1}} \, - \, 
\frac{P_{l} \frac{p}{q} + P_{l+1}}{Q_{l}\frac{p}{q}  + Q_{l+1} }
\right)
\right|  \\
{} & = & 
\left| 
\frac{y}{Q_{l+1} (Q_l y + Q_{l+1})} \, + \, 
\frac{\frac{p}{q}}{Q_{l+1} (Q_l \frac{p}{q} + Q_{l+1})}
\right| \\ 
{} & \le & 
\frac{\left| y \, - \, \frac{p}{q} \right|}{|Q_{l+1} (Q_l y + Q_{l+1})|} 
\, + \, 
\left| 
\frac{\frac{p}{q}}{Q_{l+1} (Q_l \frac{p}{q} + Q_{l+1})} \, - \, 
\frac{\frac{p}{q}}{Q_{l+1} (Q_{l} y + Q_{l+1}) } 
\right| \\
{} & \le & 
\frac{1}{|Q_{l+1} (Q_l y + Q_{l+1})|}  \left(\frac{t}{q^2} \, - \, 
                                                  \frac{\varepsilon}{2} \right)
\, + \, \left|\frac{p}{q}\right|
\left|\frac{1}{Q_{l+1} (Q_l \frac{p}{q} + Q_{l+1})} \, - \, 
\frac{1}{Q_{l+1} (Q_{l} y + Q_{l+1}) } 
\right| . 
\end{eqnarray*}
From the definition of $Q_{n+l+1}$, we see 
\[
q_{l+n+1} \, = \, p \cdot Q_{l} \, + q \cdot Q_{l+1} . 
\]
Since $|y| < 1$, $\left|\frac{p}{q}\right|$ cannot be large, 
$Q_{l+1} \, = \, \hat{a}_0 Q_{l} \, + \, \hat{\varepsilon}_{0} Q_{l-1}$, 
and $\frac{Q_{l-1}}{Q_l}$ is bounded (see \cite{B-K-S}), we see that 
\[
\left| \frac{Q_{l}y}{Q_{l+1}} \right| \quad \mbox{and} \quad 
\left| \frac{Q_{l} \frac{p}{q}}{Q_{l+1}} \right| 
\]
can be arbitrarily small and 
\[
\frac{q \cdot Q_{l+1}}{Q_{l+n+1}}
\]
can be sufficiently close to $1$ when we choose $M_1$ sufficiently large 
(note that $\hat{a}_0 = M_1 \lambda_{k}$). 

In the above discussion, the choice of $M_1$ can be independent of 
$(\varepsilon'_1, b_1), \, (\varepsilon'_2, b_2), \, \ldots \, 
(\varepsilon'_l, b_l)$. Thus we get 
\[
\left| y_0 \, - \, \frac{P}{Q} \right| \, < \, \frac{t}{Q^2}
\]
It is obvious from the construction that 
\[
\frac{P}{Q} \, \ne \, \frac{p_u}{q_u}
\]
for any $u \ge 0$.  Now we pick up a ``generic point" $w \in 
[-\frac{\lambda_k}{2}, \, \frac{\lambda_k}{2})$.  Then the ergodicity 
of $S$ w.r.t. $\mu_{k}$, we have 
\begin{eqnarray*}
{} & {} & \lim_{N \to \infty} \frac{1}{N}
\sharp \{ 1 \le u \le N \, : \, ((\varepsilon_u(w), a_u(w)) = 
(\hat{\varepsilon}_{0}, \hat{a}_{0}), \, 
(\varepsilon_{u+1}(w), a_{u+1}(w)) = \\
{} & {} & \qquad \qquad \qquad
(\varepsilon_{1}, a_{1}), \, \ldots \, , \, 
(\varepsilon_{u+n}(w), a_{u+n}(w)) = 
(\varepsilon_{n}, a_{n}), \, \ldots \, , \, \\
{} & {} & \qquad \qquad \qquad \qquad 
(\varepsilon_{u+M_0}(w), a_{u+M_0}(w)) = 
(\varepsilon_{M_0}, a_{M_0})
\}  \\
{} & = & 
\mu_{k} \left(
\{ w : \, (\varepsilon_1(w), a_1(w)) = 
(\hat{\varepsilon}_{0}, \hat{a}_{0}), 
(\varepsilon_{2}(w), a_{2}(w)) =  \right. \\
{} & {} & \qquad \qquad \qquad \left.
(\varepsilon_{1}, a_{1}), \, \ldots \, , \, 
(\varepsilon_{M_0+1}(w), a_{M_0+1}(w)) = 
(\varepsilon_{M_0}, a_{M_0}) \}
\right) \\
{} & > & 0 .
\end{eqnarray*}
Finally we look at  
\begin{eqnarray*}
{} & {} & 
\frac{1}{\ln Q} \sharp 
\{1 \le q \le Q \, : \, \left|w \, - \, \frac{p}{q} \right| \, < \, 
\frac{t}{q^2} \} \\
{} & = & 
\frac{1}{\ln Q} \sharp 
\{1 \le q \le Q \, : \, \left|w \, - \, \frac{p}{q} \right| \, < \, 
\frac{t}{q^2}, \, \frac{p}{q} = \frac{p_n}{q_n} \, \mbox{for some} \, 
n\ge 1 \} \\{} & {} & \qquad + \, 
\frac{1}{\ln Q} \sharp 
\{1 \le q \le Q \, : \, \left|w \, - \, \frac{p}{q} \right| \, < \, 
\frac{t}{q^2}, \, \frac{p}{q} \ne \frac{p_n}{q_n} \, \mbox{for some} 
\, n\ge 1 \} .
\end{eqnarray*}
From the above discussion, the second term has a positive ``liminf" and 
then the first term can not converge to $\frac{4k \lambda_{k} t}{(k-2)\pi^2}$. 
Since this estimate holds for a.e. $w$, $t$ is larger than $Le_{k}$.  
\qed

Consequently we have shown the assertion of Theorem 2. The method of the proof 
in the above shows the following generalization: \\[0.5cm]
\underline{\bf Claim} 
Suppose that $T$ is a map of an interval onto itself that induces continued 
fraction expansions for real numbers in the domain interval.  Moreover we 
assume \\
(i)  $T$ has an absolutely invariant probability measure.  \\
(ii) There exists a real number $M>0$ such that for any possible coefficient 
value $c$ larger than $M$ (or $|c| > M$), one can concatenate any 
admissible sequence after $c$ as an admissible sequence of continued 
fractions arising from $T$. \\
(iii) The Legendre constant of $T$ exists. \\
(iv)  $t_0$ in Theorem 1 is larger than the Legendre constant. \\
Then the Lenstra constant exists and the Legendre and the Lenstra constants 
are equal. \\[0.5cm]

From (1) in the proof of Proposition 3 together with Corollary 4.1 of 
\cite{B-K-S}, we have the explicit value of the 
entropy of the Rosen map. 
\\[0.5cm]
\underline{\bf Corollary} {\it 
The entropy of the Rosen map w.r.t. the absolutely 
continuous invariant probability measure is equal to 
\[
 \dfrac{C \cdot (k-2)\pi^2}{2k}
\]
with
\[
C \, = \, \left\{ \begin{array}{cl}
               \dfrac{1}{\ln \{(1 + \cos \frac{\pi}{k})/
                                      \sin \frac{\pi}{k} \}} &
                                      \text{if $k$ is even,} \\
            \\
               \dfrac{1}{\ln(1 + R)} & \text{if $k$ is odd.}
               \end{array} \right.
\]  
}
\\
{\it Acknowledgements.}  The author is grateful to Thomas A. Schmidt for 
his careful reading of the first draft of this paper, helpful advice, 
and useful comments.

\end{document}